\documentclass[12pt]{amsart}
\usepackage[utf8]{inputenc}
\usepackage{graphicx}
\usepackage{amssymb,hyperref}
\usepackage{color}

\newtheorem{theorem}{Theorem}

\newtheorem{lemma}[theorem]{Lemma}
\newtheorem{corollary}[theorem]{Corollary}

\theoremstyle{definition}

\theoremstyle{remark}
\newtheorem{remark}[theorem]{Remark}

\title{Strongly Involutive Self-Dual Polyhedra}
\thanks{$^1$Supported by  PAPIIT-UNAM under project IN109218.\\ $^2$Supported by CONACyT under 
project 166306 and  support from PAPIIT-UNAM under project IN112614.\\ $^3$Supported by CONACyT Grant 268597.\\
$^4$Supported by MATHAMSUD 18-MATH-01, Project FLaNASAGraTA and by PICS07848 CNRS}

\author[Bracho]{Javier Bracho$^1$}
\address{$^1$Instituto de Matem\'aticas, UNAM}
\email{jbracho@im.unam.mx}
\author[Montejano]{Luis Montejano$^2$}
\address{$^2$Instituto de Matem\'aticas, UNAM campus Juriquilla}
\email{luis@im.unam.mx}
\author[Pauli]{Eric Pauli$^3$}
\address{$^3$Instituto de Matem\'aticas, UNAM campus Juriquilla and Institut Montpelli\'erain Alexander Grothendieck, Universit\'e de Montpellier}
\email{eriicpc@gmail.com}
\author[Ram\'irez Alfons\'in]{Jorge Ram\'irez Alfons\'in$^4$}
\address{$^4$Institut Montpelli\'erain Alexander Grothendieck, Universit\'e de Montpellier and UMI2924 - Jean-Christophe Yoccoz, CNRS-IMPA}
\email{jorge.ramirez-alfonsin@umontpellier.fr}

\begin{document}

\begin{abstract}
A polyhedron is a graph $\mathcal{G}$ which is simple, planar and 3-connected. In this note, we classify the family of {\em strongly involutive} self-dual polyhedra. The latter is done by using a well-known result
due to Tutte characterizing 3-connected graphs. We also show that this special class of polyhedra self-duality behaves topologically as the antipodal mapping.  These self-dual polyhedra are related with several problems in convex and discrete geometry including the  V\'azsonyi problem.

\end{abstract}

\maketitle

\section{Introduction} \label{sec:intro}

A planar and 3-connected graph $\mathcal{G} = (V,E)$ can be drawn in essentially one way on the sphere or the plane. This fundamental fact is a result of the work of Withney \cite{whitney}. It tells us that we not only have the sets $V$ and $E$ defined, but that the set $F$ of faces is also determined, and furthermore the dual graph $\mathcal{G}^{*}$ is well defined. The dual graph $\mathcal{G}^{*}$ is the graph whose vertex set $V^{*}$ is the set of faces $F$ of $\mathcal{G}$, and two new vertices in $\mathcal{G}^{*}$ are connected by an edge if and only if the faces that define them are adjacent in $\mathcal{G}$.
\\

In this class of graphs, each face $f$ is determined by its \textit{boundary walk}, that is, a cyclically ordered sequence $(v_1, v_2, \cdots , v_k)$ consisting of the vertices (and the edges) that are in the closure of the region defining the face $f$ (see \cite{top}). In this sense we can say that $u$ \textit{incides} on $f$, if it is any of the elements $v_1, v_2, \cdots , v_k$ of the cycle defining the face $f$. We denote this situation simply by $u \in f$. From Steinitz's theorem (\cite{steinitz}) we know that it is the same to talk about polyhedra in the sense of convex polytopes and to talk about these graphs, so we will refer to them as \textit{polyhedra}. A \textit{polyhedron} is a graph $\mathcal{G}$ that is simple (without loops and multiple edges), planar and 3-connected. 
\\

A polyhedron $\mathcal{P}$ is said to be \textit{self-dual} if there exists an isomorphism of graphs $\tau: \mathcal{P} \rightarrow \mathcal{P}^{*}$. This isomorphism is called a \textit{duality isomorphism}. There may be several of these
duality isomorphisms and each of them is a bijection between vertices and faces of $ \mathcal{P} $, such that adjacent vertices correspond to adjacent faces. We are interested in such an isomorphism that satisfies two more properties:

\begin{enumerate}
\item For each pair $u,v$ of vertices, $u\in \tau (v)$ if and only if $v \in \tau (u)$.
\item For every vertex $v$, we have that $v \notin \tau (v)$.
\end{enumerate}

Such an isomorphism will be called a \textit{strong involution}.  If $\mathcal{P}$ is a self-dual polyhedron admitting a strong involution $\tau$, we will say that $\mathcal{P}$ is a \textit{strongly involutive polyhedron}.
\\

Strongly involutive self-dual polyhedra are very common, like for example wheels on $n$-cycles with $n$ odd and hyperwheels on $n$-cycles with $n$-even (see \cite{servatius}). In fact the relevance of strongly involutive self-dual polyhedra is partially related with the famous V\'azsonyi problem. Let $T$ be a finite set of points of diameter $h$ in Euclidean $n$-space. Characterize those sets $T$ for which the diameter is attained  a maximal number of times as a segment of length $h$ with both endpoints in $T$.  For $n=3$, Gr\"umbaum, Heppes and Straszewicz (independently) proved the following: Suppose $T$ 
has the property that the diameter is attained  a maximal number of times and denote by $V$ the intersection of the balls of radius $h$ 
centered at points of $T$. Hence, the vertex singular points of $V$ are exactly $T$ and the face structure of the singular points of the boundary of $V$ is a strongly involutive self-dual polyhedra. Indeed, this unusual 
connection between discrete and convex geometry attracted the attention of several mathematicians to this and other related problems. See, for example L\'ovasz \cite{L}, Kupitz {\em et al.}, 
\cite{vazsonyi},  Montejano and Roldán-Pensado \cite{luis y edgardo}, Montejano{ \em et al.} (\cite{reuleaux}) and the work of Bezdek {\em et al.} \cite{K}.  For more about the V\'azsonyi problem see \cite{MMO}.
\\

%

In order to have a good understanding of strongly involutive self-dual polyhedra, we will use a result due to Tutte \cite{tutte} establishing that every 3-connected graph is either a wheel (a cycle where every vertex is also connected with a central vertex $o$) or it can be obtained from a wheel by a finite sequence of two operations: adding an edge between any pair of vertices and splitting a given vertex $v$, with degree $\delta (v)\geq 4$, into two new adjacent vertices $v'$ and $v''$ in such a way that the new graph obtained is still 3-connected. \\

In the following section we briefly summarize the notions and notation in relation with the above Tutte's result restricted to the case of simple and planar graphs. In \cite{grunbaum}, Gr\"unbaum and Barnette used this idea for giving two proofs of Steinitz's Theorem.
In Section 3, we show our main result that classify the strongly involutive self-dual polyhedra. Finally, in Section 4, we give a geometric interpretation of strong involutions by proving that such a duality is topologically equivalent to the antipodal mapping on the sphere. 

\section{Tutte's Theorem for polyhedra.} \label{sec:tutte thm}

In this section we summarize the main ideas and terminology of a recursive classification of spherical polyhedra. These results are deduced from Tutte's work and the details can be found in \cite{E}. Let $\mathcal{G}$ be a polyhedron and $e=(uv)$ any edge of $\mathcal{G}$. We write $\mathcal{G}\backslash e$ for the graph obtained from $\mathcal{G}$ by deleting $e$. We write $\mathcal{G}/e$ for the graph obtained from $\mathcal{G}\backslash e$ by identifying its endpoints $u$ and $v$ in a single vertex $uv$. In the same way, given any subset $X$ of $V$, we write $\mathcal{G}\backslash X$ for the graph obtained from $\mathcal{G}$ by ommiting the elements of $X$ and any edge such that one of its endpoints is an element of $X$. We will say that $e=(uv)$ \textit{can be deleted} if $\mathcal{G}\backslash e$ is a polyhedron and we say that $e=(uv)$ \textit{can be contracted} if $\mathcal{G}/e$ is a polyhedron. We will say that $X$ is an $n-$\textit{cutting set} if it has $n$ vertices and $\mathcal{G}\backslash X$ is not connected.
\\

According to Tutte's terminology, we will say that an edge $e$ is \textit{essential} if neither $\mathcal{G}\backslash e$ nor $\mathcal{G}/e$ are polyhedra. In other words, $e$ is essential if it cannot be deleted and it cannot be contracted. 
 
\begin{theorem} \cite{E} The following statements are equivalent.
\begin{enumerate}
\item $\mathcal{G}$ is a wheel.
\item Every edge is essential.
\item Every edge is on a triangular face and has one of its endpoints of degree 3.
\end{enumerate}
\end{theorem}

This result can be rephrased as follows.

\begin{remark} Every polyhedron is either a wheel or it can be obtained by a wheel by adding new \emph{diagonal edges} within faces of the polyhedron or its dual's. Equivalently: if a polyhedron is not a wheel there is always a not essential edge, this means, an edge we can delete or contract in order to obtain a new polyhedron with one fewer edge. 
\end{remark}

In this way we can \emph{reduce} any polyhedron by a finite sequence of this operations until we obtain a wheel. It happens that one can obtain different wheels from a given polyhedron by selecting different sequences of non essential edges. 

\section{Strongly involutive polyhedra.} \label{ssd}

Throughout this section, we let $\mathcal{P}=(V,E,F,\tau)$ be a strongly involutive self-dual polyhedron and $(ab)\in E$ any edge of $\mathcal{P}$. By definition $\tau (a)$ and $\tau (b)$ are adjacent faces of $\mathcal{P}$, thus there must be an edge $(xy)\in E$ such that $\tau (a)\cap \tau (b)=(xy)$ and condition (1) of strong involution implies that $\tau (x) \cap \tau (y)=(ab)$. We will write $\tau (ab)$ for the edge $(xy)$. We will say that $(ab)$ is a \textit{diameter} if and only if $a\in \tau (b)$ (and therefore $b\in \tau (a)$). 

\begin{lemma} If $(ab)$ and $(xy)$ are both diameters, then $\mathcal{P}$ is the tetrahedron $K_4$.
\end{lemma}

\begin{proof} From the hypoteses we deduce $a\in \tau (x)\cap \tau (y) \cap \tau (b)$ and $x\in \tau (a)\cap \tau (b) \cap \tau (y)$, then $\{a,x\}\subset \tau (y)\cap \tau (b)$ but from the 3-connectivity, the intersection of any two faces must be empty, a single vertex or a single edge, thus $(ax)$ is an edge, otherwise $\{a,x\}$ is a 2-cutting set. Analogously, $(bx)$ is an edge, otherwise $\{b,x\}$ is a 2-cutting set. In the same way, $(ya)$ and $(yb)$ are edges and thus $\mathcal{P}$ is $K_4$.
\end{proof}

\begin{lemma} If $(ab)$ is a diameter and $(xy)$ is not, then $\{a,b,x\}$ and $\{a,b,y\}$ are 3-cutting sets of $\mathcal{P}$.
\end{lemma}

\begin{proof} In view of the 3-connectivity and thus of Withney's Theorem, $\tau (a), \tau (b), \tau (x)$ and $\tau (y)$ are faces homeomorphic to disks and from the hypoteses we can deduce that $\tau (x) \cup \tau(y)$ and $\tau (a) \cup \tau (b)$ are also disks such that $(\tau (x) \cup \tau(y))\cap (\tau (a) \cup \tau (b))=\{a,b\}$, then we can observe that $(\tau (x) \cup \tau(y))\cup (\tau (a) \cup \tau (b))$ has the homotopy type of a circle and thus its complement consists of two regions $R_1$ and $R_2$.
Let's suppose $x\in \overline{R_1}$ and $y\in \overline{R_2}$. Let $u$ and $w$ be vertices in $\tau (x) \cup \tau (y) \setminus \{a,b\}$ such that $u\in \overline{R_1}$ and $w\in \overline{R_2}$. Then in $\mathcal{P}\backslash \{a,b,x\}$, $u$ and $y$ are disconnected and in $\mathcal{P}\backslash \{x,y,b\}$, $w$ and $x$ are disconnected.
\end{proof}

\begin{theorem} If $\mathcal{P}$ is not a wheel, then there exists an edge $e$ satisfying the three following conditions:

\begin{enumerate}
\item It is not on a triangular face.
\item It is not in a 3-cutting set.
\item It is not a diameter.
\end{enumerate}
\end{theorem}

\begin{proof} By the previous two lemmas, it is enough to show that there exists an edge not on a triangular face and not in a 3-cutting set. By combining Theorem 1 and the self-duality of $\mathcal{P}$ we ensure the existence of such an edge.
\end{proof}

\begin{theorem} Let $e=(ab)$ be an edge which is neither on a triangular face nor in a 3-cutting set nor
a diameter. Then, the graph $[\mathcal{P}/(ab)]\backslash \tau (ab)$, denoted by $\mathcal{P}^\diamond=\mathcal{P}^\diamond_{ab}$, is a strongly involutive self-dual polyhedron.
\end{theorem}

\begin{proof} Since $(ab)$ satisfies the three properties of last theorem, then $\mathcal{P}/(ab)$ is a polyhedron, and therefore its dual $\mathcal{P}\backslash \tau (ab)$ is also a polyhedron. We will show that $\mathcal{P}^\diamond$ is a polyhedron. Indeed, it is simple and planar. We need it to be 3-connected. If it were not, then it would have a 2-cutting set $\{m,n\}$. Since $\tau (a)$ and $\tau (b)$ are the faces such that $\tau (a) \cap \tau (b)= \tau (ab)$ we may observe that one of the elements in $\{m,n\}$ is in $\tau (a)$ and the other is in $\tau (b)$. Let's supose $m\in \tau (a)$ and $n\in \tau (b)$. Furthermore the vertex $a=b$, denoted by $ab$ must be one of the elements in $\{m,n\}$, otherwise $\{m,n\}$ would be a 2-cutting set of $\mathcal{P}\backslash \tau (ab)$, a contradiction. This implies that in $\mathcal{P}$, $a\in \tau (b)$ (and therefore $b\in \tau (a)$), so $(ab)$ would be a diameter, which is not by hypothesis. Finally, by definition, $\mathcal{P}^\diamond$ is self-dual and it is strongly involutive with isomorphism $\tau^\diamond (u)=\tau(u)$ for every $u\notin\{a,b\}$ and with $\tau^\diamond (a=b)$ the face obtained by the union of $\tau(a)$ and $\tau(b)$ when edge $(xy)$ is deleted.\end{proof}

By the above theorem, we can define the \emph{remove-contract} operation in any strongly involutive polyhedron not a wheel: there is at least one edge $(ab)$ that we can contract and at the same time remove the edge $\tau (ab)$ in order to obtain a new strongly involutive polyhedron. We can apply this operation repeatedly in order to finish with a strongly involutive wheel (with odd number of vertices in the main cycle). Conversely, we can start with such a wheel and then diagonalizing faces and splitting their corresponding vertices carefully in order to \emph{expand} a strongly involutive polyhedron. By diagonalizing we mean that given a face that is not a triangle, we add a new edge within the face joining two non-consecutive vertices. 

In the above terms, Theorem 6 gives the following. 

\begin{corollary} Every strongly involutive self-dual polyhedra is either a wheel or it can be obtained from an odd wheel by a finite sequence of operations consisting in diagonalizing faces of the polyhedron and its dual's simultaneously. \end{corollary}

\section{Topological interpretation.} \label{top int}

In this section we are going to consider topological embeddings of a given graph $\mathcal{G}$ on the surface $\mathbb{S}^2$. By Whitney's Theorem we know that if $\mathcal{G}$ is simple, planar and 3-connected, then any two such embeddings are equivalent in the sense that the set of faces (and their adjacencies) is fully determined just by the embedding of the graph. It is an interesting fact that with these conditions we can choose one of these embeddings in such a way that any automorphism of the graph of $\mathcal{P}$ acts as an isometry of $\mathbb{S}^2$. We will write this important fact as follows.

\begin{lemma}\cite[Lemma 1]{servatius} \textbf{(Isometric embedding lemma.)} There exists an embedding $i:\mathcal{G}\to \mathbb{S}^2$ such that for every $\sigma \in \mbox{Aut}(\mathcal{G})$ there exists an isometry $\tilde{\sigma}$ of $\mathbb{S}^2$ satisfying $i\circ \tilde{\sigma}=\sigma \circ i$. \end{lemma}

Our goal for now is to interpretate geometrically the strong involutions. In the rest of the section $\mathcal{G}$ is the underlying graph (simple, planar and 3-connected) of a strongly involutive self-dual polyhedron $\mathcal{P}$. 
\\
Let us define $\mathcal{G}_{\square}$ the \emph{graph of squares of $\mathcal{G}$} as follows: 
$$\begin{array}{ll}
V(\mathcal{G}_{\square}) & = V(\mathcal{G}) \cup F(\mathcal{G}) \cup A(\mathcal{G}) \text{ and}\\
E(\mathcal{G}_{\square}) & =  \{(ve):v\in V(\mathcal{G}), e\in E(\mathcal{G}), v\in e\} \cup \\
& \ \ \ \{(ec): e\in E(\mathcal{G}), f\in F(\mathcal{G}), e\in f\}.
 \end{array}$$
 
It is easy to observe that $\mathcal{G}_{\square}$ is a 3-connected simple planar graph and therefore it can be drawn on the sphere in such a way that any automorphism of $\mathcal{G}_{\square}$ is an isometry. We can suppose $\mathcal{G}_{\square}$ is embedded in that way and we will abuse of notation making no distintion between $\mathcal{G}_{\square}$ and its image under the embedding. By definition, the faces of $\mathcal{G}_{\square}$ are all quadrilaterals of the form $(vafb)$, where $v\in V(\mathcal{G}), a,b\in E(\mathcal{G})$ and $f\in F(\mathcal{G})$.  

\begin{theorem} Let $\tau$ be a strong involution of $\mathcal{P}$. Then $\tilde{\tau}$ is the antipodal mapping 
$\alpha :\mathbb{S}^2 \rightarrow \mathbb{S}^2$, $\alpha (x)=-x$.
\end{theorem}
\begin{proof} First we can observe that $\tau$ is an automorphism of $\mathcal{G}_{\square}$ and condition (1) of strong involution implies $\tau ^2=id$. Therefore, $\tilde{\tau}$ (given in Lemma 8) must be an involution as isometry. There are three possible involutive isometries of the sphere: a reflection through a line (a spherical line), a rotation by $\frac{\pi}{2}$ and the antipodal mapping (a good reference is \cite{geometry}). Only the antipodal mapping has no fixed points, so we will show that $\tilde{\tau}$ cannot have fixed points. We will proceed by contradiction, supposing $\tilde{\tau}$ has a fixed point and then we will conclude there exists a vertex $v$ such that $v\in \tau (v)$.
\\
If $\tau$ is a reflection through a plane $H$, let us consider $v\in V(\mathcal{G})$, $a,b\in E(\mathcal{G})$ and $f\in F(\mathcal{G})$ such that $H$ intersects the quadrilateral $\mathcal{Q} = (vafb)$ in its interior. The only points of the edges of quadrilateral $\mathcal{Q}$ can intersect $H$ are $a$ and $b$, so $H\cap \mathbb{S}^2 = l$ where $l$ is the spherical line through $a$ and $b$, thus we must have $\tau (v)=f$ that means $v\in \tau (v)$.
\\
If $\tau$ is a rotation in a line $PP'$ ($P,P'$ antipodal points on the sphere), let $\mathcal{Q} = (vafb)$ be a quadrilateral containing $P$. If $P$ is the center (the barycenter) of the quadrilateral, then since $\tau$ is a duality, it must send $v$ into $f$, but then $\tau (v)=f$, which means $v\in \tau (v)$. If $P$ is $a$ or $b$, say $P=a$, then the edge $(va)$ is sent to an edge $(af')$ where $f'$ is a face of $\mathcal{G}$, distinct from $f$ and containing $v$, but then the quadrilateral $\mathcal{Q'}$ corresponding to $v$ and $f'$ we have $\tau (v) = f'$, which means $v\in \tau (v)$. This concludes the proof.  
\end{proof}
 
 As a consequence of Theorem 9 we obtain the following.
 
\begin{corollary} For a strongly involutive self-dual polyhedron there is only one duality which is a strong involution.
\end{corollary}


\begin{thebibliography}{BLNP07}

\bibitem[1]{grunbaum} D. W. Barnette, B. Grünbaum, \emph{On Steinitz's Theorem concerning convex 3-polytopes}, Lecture Notes in Mathematics \textbf{110} (1968),  27--40-

\bibitem[2]{K}
K. Bedzdek, Z. L\'angi, M. Nasz\'odi, P. Papez, \emph{Ball-polyhedra},  Discret. Comput. Geom. \textbf{38} (2007), 201-230.

\bibitem[3]{geometry} D. A. Brannan, M. F. Esplen and J.J. Gray. \emph{Geometry} - 2nd ed., Cambridge, 2012.  

\bibitem[4]{E}
Eric Pauli, \emph{Poliedros autoduales fuertemente involutivos}, Tesis Doctoral. UNAM. 2020

\bibitem[5]{top} J. L. Gross and T. W. Tucker. \emph{Topological Graph Theory}, Wiley Interscience,
New York, 1987.

\bibitem[6]{vazsonyi}
Y.~S. Kupitz, H.~Martini, and M.~A. Perles, \emph{Ball polytopes and the
  V{\'a}zsonyi problem}, Acta Mathematica Hungarica \textbf{126} (2010),
  no.~1-2, 99--163.
  
  \bibitem[7]{L}
  L. L\'ovasz, \emph{Self-dual polytopes and the chromatic number os distance graphs on the sphere,}  Acta Sci. Math. \textbf{45} (1983), 317-323.
  
  \bibitem[8]{MMO} 
  H.~Martini, L. Montejano and D. Oliveros.  Bodies of Constant width; An introduction to convex geometry with applications. Birkh\"user. 2019. 

\bibitem[9]{luis y edgardo}
L.~Montejano and E.~{Rold{\'a}n-Pensado}, \emph{Meissner polyhedra}, Acta
  Mathematica Hungarica \textbf{151} (2017), no.~2, 482--494.

\bibitem[10]{reuleaux} 
L. Montejano, E. Pauli, M. Raggi and E. Roldán-Pensado, \emph{The Graphs Behind Reuleaux Polyhedra}, Preprint (2019), arXiv:1904.12761v1[cs.CG].
  
\bibitem[11]{tutte} W. T. Tutte, \emph{A Theory of 3-connected Graphs.} 1961.

\bibitem[12]{servatius}
B. Servatius and H. Servatius, \emph{The 24 Symmetry pairings of self-dual maps on the sphere}, Discrete Mathematics \textbf{140} (1995), no.~1-3, 167--183.

\bibitem[13]{steinitz}
E. Steinitz, \emph{Polyeder and Raumeinteilungen}, Enzykl. Math. Wiss. 3 (Geometrie) \textbf{3AB12} (1922), 1--139.  

\bibitem[14]{whitney}
H. Whitney, \emph{2-Isomorphic graphs}, Amer. J. Math. \textbf{55} (1933), 245--254.  



\end{thebibliography}
\end{document}